\documentclass[11pt,reqno]{amsart}

\usepackage[T1]{fontenc}
\usepackage{lmodern}
\usepackage{microtype}
\usepackage{amsmath,amssymb,amsthm,amsfonts,mathtools,mathrsfs}
\usepackage[breaklinks=true,pdfborder={0 0 0}]{hyperref}
\usepackage{url}
\usepackage[msc-links, lite, alphabetic]{amsrefs}
\usepackage{geometry}
\theoremstyle{plain}
\newtheorem{theorem}{Theorem}[section]
\newtheorem{proposition}[theorem]{Proposition}
\newtheorem{lemma}[theorem]{Lemma}

\newtheorem{conjecture}[theorem]{Conjecture}

\theoremstyle{definition}

\theoremstyle{remark}
\newtheorem{remark}[theorem]{Remark}

\numberwithin{equation}{section}

\DeclareMathOperator{\Ric}{Ric}
\DeclareMathOperator{\Scal}{Scal}
\DeclareMathOperator{\vol}{vol}

\DeclareMathOperator{\diam}{diam}
\DeclareMathOperator{\tr}{tr}

\newcommand{\R}{\mathbb R}

\newcommand{\haus}{\mathcal H}

\newcommand{\eps}{\varepsilon}
\newcommand{\dist}{d}
\newcommand{\defeq}{\mathrel{\mathop:}=}
\newcommand{\dd}{\,d}

\newcommand{\ip}[2]{\langle #1,#2\rangle}

\def\Xint#1{\mathchoice
{\XXint\displaystyle\textstyle{#1}}%
{\XXint\textstyle\scriptstyle{#1}}%
{\XXint\scriptstyle\scriptscriptstyle{#1}}%
{\XXint\scriptscriptstyle\scriptscriptstyle{#1}}%
\!\int}
\def\XXint#1#2#3{{\setbox0=\hbox{$#1{#2#3}{\int}$ }
\vcenter{\hbox{$#2#3$ }}\kern-.6\wd0}}

\def\dashint{\Xint-}

\title[Codimensional volume growth]
{Positive Scalar Curvature and Volume Growth}
\author{Bochao Kong}
\address[Bochao Kong]{Michigan State University, East Lansing, MI, USA}
\email{kongboc1@msu.edu}
\author{Xingyu Zhu}
\address[Xingyu Zhu]{Michigan State University, East Lansing, MI, USA}
\email{zhuxing3@msu.edu}
\date{}

\hypersetup{
  pdftitle={Positive Scalar Curvature and Volume Growth},
  pdfauthor={Bochao Kong and Xingyu Zhu}
}

\subjclass[2020]{53C21, 53C23, 58J35}
\keywords{nonnegative Ricci curvature, positive scalar curvature, volume growth, heat kernel, Wasserstein distance, Liouville measure}

\begin{document}

\begin{abstract}
For a complete Riemannian manifold with nonnegative Ricci curvature, we prove two sharp volume growth order estimates, thereby resolving a conjecture of Gromov from 1986. The first states that a uniform deficit in the volume of unit balls, an analog of positive macroscopic scalar curvature, forces codimension one volume growth, while the second states that a uniformly positive scalar curvature lower bound forces codimension two growth, as predicted by the codimension two volume growth conjecture. 
\end{abstract}

\maketitle
\tableofcontents

\section{Introduction}\label{sec:intro}
\subsection{Background and Statements}
In this paper we study large-scale effects of positive scalar curvature. Let $n\ge 2$, and let $(M^n,g)$ be a smooth Riemannian $n$-manifold. A well-known geometric consequence of scalar curvature is that it constrains the volume of small balls. At a point $p\in M$, the following expansion holds.
\begin{equation}\label{eq:small-ball-expansion}
  \vol_g B(p,r)
  =\omega_n r^n\left(
    1-\frac{\Scal_g(p)}{6(n+2)}r^2+O(r^4)
  \right),
  \qquad r\to0^+,
\end{equation}
where $\omega_n$ is the volume of the Euclidean unit $n$-ball. This description is local, and by itself does not give global control. A central and long-standing problem in scalar curvature is to understand the large-scale geometric and topological effects of scalar curvature.

For noncompact manifolds, Gromov proposed a program based on large Riemannian manifolds \cite{Gro85}. Set
\[
  V(M,r)\defeq\sup_{p\in M}\vol_g B(p,r).
\]
The manifold $(M,g)$ is called \emph{large} if
$V(M,r)\geq\omega_n r^n$ for every $r\geq1$. Failure of largeness is a
macroscopic analogue of positive scalar curvature. Later a more precise definition of macroscopic scalar curvature was given by Guth \cite{Guth2010Metaphors}. Gromov then proposed two sets of conjectures under the same folklore philosophy that manifolds having uniformly positive macroscopic scalar curvature have large-scale codimension $1$ and manifolds having uniformly positive scalar curvature have large-scale codimension $2$. To introduce the first set of conjectures we recall the definition of Urysohn width and macroscopic dimension.

For a metric space $(X,d)$ and an integer $k\geq0$, its Urysohn $k$-width is
\begin{equation}\label{eq:urysohn-width}
  \operatorname{UW}_k(X)
  =\inf_f\ \sup_{y\in P^k}\diam f^{-1}(y),
\end{equation}
where the infimum is taken over continuous maps $f:X\to P^k$ to
$k$-dimensional simplicial complexes \cite{Gromov1988Widths}. The macroscopic dimension of $X$ is the least $k$ for which
$\operatorname{UW}_k(X)<\infty$. The first set of codimension conjectures is stated as follows \cite{Gro85}.

\begin{conjecture}[Macroscopic dimension]\label{conj:macro}
Let $(M^n,g)$ be complete and noncompact.
\begin{enumerate}
\item \label{item:codim1} If, for some $r\geq1$, it holds
  $V(M,r)< \omega_n r^n$, then
  $\operatorname{UW}_{n-1}(M)\le r$.
\item \label{item:codim2} \cite{Gro85}*{2.A.(c)} There is a dimensional constant $C(n)$ such that
  $\Scal_g\geq\sigma^2>0$ implies
  $\operatorname{UW}_{n-2}(M)\leq C(n)/\sigma$.
\end{enumerate}
\end{conjecture}

Guth \cite{Guth2017UrysonWidth} and Papasoglu \cite{Papasoglu2020UrysonWidth} made substantial progress on \eqref{item:codim1} of Conjecture \ref{conj:macro}. They showed, by different methods, that there exists a universal constant $\delta(n)>0$, such that if $V(M,r)\le \delta(n) r^n$ then $\operatorname{UW}_{n-1}(M)\le r$, while the constant $\delta(n)$ is potentially much smaller than the predicted optimal constant $\omega_n$. In contrast, Kumar--Sen constructed, in dimensions at least four, closed manifolds with
uniformly positive macroscopic scalar curvature and unbounded $\operatorname{UW}_{n-2}$ \cite{KumarSen2026UrysohnWidth}, showing that the conjectured codimension $1$ phenomenon for positive macroscopic scalar curvature is sharp.  For \eqref{item:codim1} of Conjecture \ref{conj:macro}, some equivariant (i.e. universal cover of a closed manifold) macroscopic-dimension results are known under additional topological and coarse-geometric hypotheses
\cites{BolotovDranishnikov2009ScalarCurvature,Dranishnikov2013VirtualDuality}.

The second set of conjectures requires nonnegative Ricci curvature. Under nonnegative Ricci curvature, Bishop--Gromov comparison gives
\begin{equation}\label{eq:euclidean-volume-upper-bound}
  \vol_g B(p,r)\leq\omega_n r^n
  \qquad(p\in M,\ r>0).
\end{equation}
In particular $(M,g)$ has polynomial volume growth order bounded above by its (topological or Hausdorff) dimension. It is then natural to measure large-scale dimension by the polynomial order of volume growth instead of macroscopic dimension.
\begin{conjecture}[Volume Growth \cite{Gro85}*{2.A. (b)}]\label{conj:vol}
Let $(M^n,g)$ be complete and noncompact with $\Ric\ge 0$.
\begin{enumerate}
\item \label{item:codim1vol} If for some $\epsilon>0$, it holds $V(M,1)\leq\omega_n-\epsilon$, then there is a finite constant $C=C(n,\epsilon)$ such that, for every
$o\in M$ and every $R>0$, $\vol_g B(o,R)\leq C R^{n-1}$.
\item \label{item:codim2vol} There is a dimensional constant $C(n)$ such that
 if $\Scal_g\geq\sigma>0$, then, for every $o\in M$ and every $R>0$,
 $\vol_g B(o,R)\leq C(n)\sigma^{-1}R^{n-2}$.
\end{enumerate}
\end{conjecture}

We confirm the volume growth conjecture.

\begin{theorem}\label{thm:codim-one}
\eqref{item:codim1vol} of Conjecture \ref{conj:vol} is true.
\end{theorem}

\begin{theorem}
\label{thm:codim-two}
\eqref{item:codim2vol} of Conjecture \ref{conj:vol} is true.
\end{theorem}
Some historical remarks on this conjecture are in order.
\begin{remark}
    Under the stronger assumption $\sec\geq0$, Gromov claimed both codimensional volume growth bounds and outlined a splitting at infinity argument \cite{Gro85}. The more detailed proof for codimension two volume growth under positive scalar curvature was later given by Petrunin \cite{Petrunin_ScalarIntegral} and extended by Li \cite{Li2026CurvatureMeasure}. We will give a short proof of the codimension $1$ volume growth for nonnegatively curved Alexandrov spaces with unit ball volume deficit in Appendix \ref{thm:Gromov}.
    
    Although the two sets of conjectures \ref{conj:macro} and \ref{conj:vol} are under the same philosophy, the relation between volume growth order and macroscopic dimension for noncompact manifolds with nonnegative Ricci curvature remains unclear. For closed nonnegatively curved spaces, Perelman showed that volume is comparable with the product of all Urysohn widths, as conjectured by Gromov
\cites{Gromov1988Widths, Perelman1995Widths}. See \cite{YuZhu2024DimensionBounds} for a discussion on the mixture of both conjectures \ref{conj:macro} and \ref{conj:vol}, where nonnegative Ricci curvature and positive scalar curvature together bound the Assouad--Nagata dimension, a notion of large scale dimension stronger than macroscopic dimension.

\eqref{item:codim2vol} of Conjecture \ref{conj:vol} has been extensively studied in dimension $3$. Chodosh--Li--Stryker \cite{OLS23}, Munteanu--Wang \cite{MunteanuWang2022Geometry} showed linear volume growth with non-universal constants, B. Zhu-- X. Zhu\cite{ZhuZhu23} studied how the positive scalar curvature pass to Gromov Hausdorff limit and Wang \cite{YipengWang2024Effective} provided an effective and universal constant for the linear volume growth. Much less was known in higher dimensions. Some partial results in higher dimensions were obtained by B. Zhu \cite{Zhu_Geometryofpsc}, Wang--Xie--B. Zhu--X. Zhu  \cite{WangXieZhuZhu24}, X. Zhu \cite{XingyuZhu2023TwoDimension}, and Cucinotta--Mondino \cite{CucinottaMondino26}.

Cai \cite{Cai_Largemanifold} claimed to have proved \eqref{item:codim1vol} of Conjecture \ref{conj:vol}. However, there are several details that could not be verified. 
\end{remark}

\begin{remark}
    While preparing the draft, we learned of independent work by Antonelli \cite{antonelli2026universalvolumegrowthbounds} which proves Theorem \ref{thm:codim-two} and shows corresponding volume growth estimates under intermediate positive curvature assumptions. Although the methods of \cite{antonelli2026universalvolumegrowthbounds} is quite different from ours, the spirit of relating positive curvature bounds with number of almost splitting factors remains the same. Compare Lemma \ref{c2:lem:local-splitting} and \cite{antonelli2026universalvolumegrowthbounds}*{Proposition 3.1}. Another independent work of Koirala \cite{koirala2026volumegrowthpositiveintermediate} also proved similar results and local volume estimates under general Ricci lower bounds using heat kernel methods. This line of thoughts is then further developed by Han \cite{han2026heatcodimensionscalarcurvature}. 
\end{remark}

\subsection{Strategy of proof}
Several new observations emerge in the proof. First, we observe that there is a rigidity associated to the heat kernel pullback of the $2$-Wasserstein metric introduced by Gigli--Mantegazza \cite{GigliMantegazza2014HeatKernel}. Compare the pullback metric from $L^2(M)$ \cites{BBG94,AHPT}. At time $t>0$, the heat-kernel embedding of $(M,g)$ into the Wasserstein space $(\mathcal P_2(M),W_2)$ determines a Riemannian metric $G_t\le g$. An equivalent definition we use is
\begin{equation}\label{eq:defGt}
G_{t,x}(v,v)\defeq\sup_{\psi\in C_c^\infty(M)}
  \left\{2\langle\nabla P_t\psi(x),v\rangle-P_t|\nabla\psi|^2(x)\right\}, \quad x\in M, v\in T_xM.
\end{equation}
We will consider the deficit $D_t\defeq g-G_t$.

The new rigidity phenomenon we observe is that the number of zero eigenvalues of $D_t$ gives rise to the same number of splitting functions, hence $M$ splits off the same number of $\R$-factors. This is in fact the rigidity of the equality case of $2$-Bakry-Emery inequality, which is briefly discussed by Ambrosio--Brue--Semola \cite{AmbrosioBrueSemola2019Rigidity}. To explain the idea we assume that there exists a point $x\in M$ at which the zero eigenvalue of $D_t$ has multiplicity $2$. Let $e_1$, $e_2$ be an orthonormal basis spanning the eigenspace in $T_xM$. For illustrative purposes and to avoid technical complexity, we assume the supremum in \eqref{eq:defGt} is attained by a smooth function for every $v\in \mathrm{span}\{e_1,e_2\}$. Then there are smooth functions $\psi_1$ and $\psi_2$ so that for $i=1,2$
\begin{align*}
    0&=g(e_i,e_i)-G_{t,x}(e_i,e_i)\\
    &=|e_i|^2-2\langle \nabla P_t\psi_i(x),e_i\rangle+P_t(|\nabla \psi_i|^2)(x)\\
    &=|\nabla P_t\psi_i(x)-e_i|^2+P_t(|\nabla \psi_i|^2)(x)-|\nabla P_t\psi_i(x)|^2\ge 0.
\end{align*}
We see that the vectors $\nabla P_t\psi_i(x)=e_i$ are orthogonal for $i=1,2$, and equality in the $2$-Bakry-Emery inequality holds:
$P_t(|\nabla \psi_i|^2)(x)=|\nabla P_t\psi_i(x)|^2$. Then taking the derivative of $F_i(s)=P_s(|\nabla P_{t-s} \psi_i|^2)(x)$ and applying the Bochner formula yields that $P_t\psi_i$ are both harmonic splitting functions, together inducing a splitting $\R^2$-factor. The complete proof of a slightly more general almost-rigidity statement will be given in Proposition \ref{prop:rank-k-heat-frame}.

In particular, if all eigenvalues of $D_t$ are zero at a point, equivalently, $\tr_g G_t=\tr_g g=n$, then $M$ will be isometric to $\R^n$. See Theorem~\ref{thm:maximal-trace-almost-rigidity}.

This (almost) rigidity will be applied to the two main theorems. In the setting of Theorem \ref{thm:codim-one}, the unit ball volume deficit implies that no $n$-almost splitting can happen at scale at least $1$, which in turn implies $D_t$ must have uniformly positive trace at every point, hence at least one uniformly positive eigenvalue at every point. In the setting of Theorem \ref{thm:codim-two}, another new observation is that uniformly positive scalar curvature $\Scal_g\ge 1$ will prevent $(n-1)$-almost splitting at large scale. This was previously achieved only under an additional volume noncollapsed assumption \cites{WangXieZhuZhu24,CucinottaMondino26}. We will explain why we can drop the noncollapsed assumption around Lemma \ref{c2:lem:local-splitting}. Combined with the splitting rigidity of $D_t$, $\Scal_g\ge 1$ implies that at least two eigenvalues of $D_t$ are positive.

Finally we will turn the number of positive eigenvalues into the loss in volume growth order. Note that the minimal number of positive eigenvalues of $D_t$ matches with the codimension in each setting. This relation is revealed in Lemma \ref{lem:angular-k}. Heuristically, $D_t$ detects the directions in which geodesics cannot remain long on average. Indeed, consider a minimizing geodesic $\gamma:[-3r/2,3r/2]\to M$. If the largest eigenvalue of $D_t$ has a uniform positive lower bound $\rho>0$ along $\gamma$, with corresponding unit eigenvector $v_{\gamma(s)}$, then Proposition \ref{prop:centered-segment}, which gives an upper bound for the distance distortion along $\gamma$, yields
\[
 \int_{-r}^r\rho\left|\langle\dot\gamma(s),v_{\gamma(s)} \rangle\right|^2\, ds\le \int_{-r}^{r}D_t(\dot\gamma,\dot\gamma)\,ds
  \leq\frac{8nt}{r}.
\]
When $r$ is large while $t,\rho>0$ are fixed, we see that, on average, $\dot\gamma$ is nearly perpendicular to the eigen-direction. It suggests the positive eigenvalue directions will contribute much less to the volume growth of a geodesic ball. To apply this idea to all geodesics, a geodesic flow argument and a pairing are needed, which are the contents of Proposition \ref{prop:liouville-tail-k} and \ref{prop:ordered-pairs-k}.

\begin{remark}
    The constant in front of the volume growth power depends on the positive lower bound of the eigenvalues of $D_t$, which is provided by an argument by contradiction. However, this constant is effective. It is possible to compute the lower eigenvalue bound based on scalar curvature lower bound.  
\end{remark}

\medskip
\noindent\textbf{Acknowledgements.} 
X.Z is supported by AMS Travel Fund.

\noindent\textbf{AI disclosure.}
Generative AI tools, more explicitly, ChatGPT 5.6 Sol Ultra and Codex, assisted with proof exploration, organization, and drafting. Essential ideas were generated by AI. The authors reviewed, edited, and remain responsible for all theorem statements, proofs and citations. In particular, most of the proofs are different from the original proof generated by AI. The heat pullback metric idea and the Liouville measure averaging trick are AI insights. The Appendix does not have AI use other than text editing.  

During the preparation of this draft, we became aware of the preprint \cite{ge2026heatkernelgeometrygromovs}, in which an independent proof of Conjecture~\ref{conj:vol} was announced using very similar ideas and was reportedly developed without AI assistance. The AI-assisted proof presented in this article was produced before \cite{ge2026heatkernelgeometrygromovs} became publicly available, and we did not use the aforementioned preprint as a source in our AI-assisted work.

\section{The heat-transport deficit}\label{sec:heat-transport}
In this section we review \cite{GigliMantegazza2014HeatKernel} and prove some distortion estimates of the pullback of the Wasserstein metric. Let $P_t$ denote the heat semigroup for $\partial_tu=\Delta u$, with heat
kernel $p_t$, and put
\[
  \mu_{t,x}=p_t(x,\cdot)\,d\vol_g.
\]
On a complete manifold with nonnegative Ricci curvature, this is a probability measure with finite second moment, i.e. $\mu_{t,x}\in \mathcal{P}_2(M)$. 
For $v\in T_xM$, \cite{GigliMantegazza2014HeatKernel} defined a pullback metric
\begin{equation}\label{eq:heat-metric}
  G_{t,x}(v,v)
  :=\sup_{\psi\in C_c^\infty(M)}
  \left\{2\langle\nabla P_t\psi(x),v\rangle
             -P_t|\nabla\psi|^2(x)\right\},
  \qquad D_t:=g-G_t.
\end{equation}
Here we used the variational form in Gigli--Mantegazza \cite{GigliMantegazza2014HeatKernel}*{Proof of Proposition 3.7}. It is a pullback metric by the following observation. Fix $t>0$, let $\iota_t: M\to \mathcal{P}_2(M)$, such that $\iota_t(x)=\mu_{t,x}$. For every $\mu\in \mathcal{P}_2(M)$ there is a tangent space 
\[\mathrm{Tan}_{\mu}\mathcal{P}_2(M)=\overline{\{\nabla\phi:\phi\in C_c^\infty(M)}\}^{L^2(\mu)}\]
with an inner product
\[
g_{W_2,\mu}(V,W)=\int_M \langle V,W\rangle\, d\mu,
\]
then $G_t=\iota_t^*g_{W_2}$.
The variational form is more suitable as it involves the Bakry--Emery contraction estimates.

\begin{proposition}\label{prop:heat-transport}
Let $(M,g)$ be complete with $\Ric\geq0$. For every $t>0$, $G_t$ defined in
\eqref{eq:heat-metric} is a Riemannian metric satisfying
\begin{equation}\label{eq:defect-positive}
  0\leq G_t\leq g.
\end{equation}
If $\gamma:[a,b]\to M$ is a smooth curve, then
\begin{equation}\label{eq:wasserstein-upper}
  W_2(\mu_{t,\gamma(a)},\mu_{t,\gamma(b)})
  \leq\int_a^b\sqrt{G_t(\dot\gamma,\dot\gamma)}\,ds.
\end{equation}
\end{proposition}

\begin{proof}
Take any $\psi\in C_c^\infty(M)$. The Bakry--\'{E}mery estimate
\begin{equation}\label{eq:bakry}
  |\nabla P_t\psi|^2\leq P_t|\nabla\psi|^2
\end{equation}
holds under nonnegative Ricci curvature
\cite{AmbrosioGigliSavare2015BakryEmery}.  For fixed $x$ and $v$, it is straightforward to verify that 
\[
2\langle \nabla P_t\psi(x),v\rangle-P_t(|\nabla \psi|^2)(x)=-|\nabla P_t\psi(x)-v|^2+|v|^2+|\nabla P_t\psi(x)|^2-P_t(|\nabla \psi|^2)(x)\le |v|^2.
\]
Taking supremum over $\psi$ gives $G_t\le g$. Taking $\psi=0$ gives $G_t\ge 0$. Notice that  \cite{GigliMantegazza2014HeatKernel}*{Proposition 3.5} does not need compactness, which is the setting of \cite{GigliMantegazza2014HeatKernel}, as long as the measures $\mu_{t,\gamma(s)}\in \mathcal P_2(M)$. Then it implies the Wasserstein curve $s\mapsto\mu_{t,\gamma(s)}$ is absolutely continuous in $W_2$ metric and 
\[
G_t(\gamma_s',\gamma_s')=\left|\dot\mu_{t,\gamma(s)}\right|^2\quad \text{a.e. $s\in[a,b]$}.
\]
Then it follows
\[
 W_2(\mu_{t,\gamma(a)},\mu_{t,\gamma(b)})
  \leq\int_a^b \left|\dot\mu_{t,\gamma(s)}\right|\, ds=\int_a^b\sqrt{G_t(\dot\gamma,\dot\gamma)}\,ds,
\]
as desired.
\end{proof}

Given that $G_t\le g$, we can compare the length of a curve in the length metrics induced by both $G_t$ and $g$. The next proposition measures the distance distortion in the interior of a geodesic w.r.t. $g$.
\begin{proposition}\label{prop:centered-segment}
Fix $t,r>0$, and let $\gamma:[-3r/2,3r/2]\to M$ be a unit-speed minimizing geodesic.  Then
\begin{equation}\label{eq:centered-estimate}
  \int_{-r}^{r}D_t(\dot\gamma,\dot\gamma)\,ds
  \leq2\bigl(\sqrt{r^2+8nt}-r\bigr)
  \leq\frac{8nt}{r}.
\end{equation}
\end{proposition}

\begin{proof}
Put
\[
  p=\gamma(-3r/2),\quad x=\gamma(-r),\quad
  y=\gamma(r),\quad q=\gamma(3r/2).
\]
Laplacian comparison gives, for $d_a=d(a,\cdot)$,
\begin{equation}\label{eq:squared-distance-heat}
  P_t(d_a^2)(z)\leq d(z,a)^2+2nt.
\end{equation}
Since $d_p+d_q\geq d(p,q)=3r$, Kantorovich duality of $1$-Wasserstein distance,
\eqref{eq:squared-distance-heat}, and Cauchy--Schwarz yield
\begin{align}
  W_2(\mu_{t,x},\mu_{t,y})
  \geq W_1(\mu_{t,x},\mu_{t,y})\geq P_td_p(y)-P_td_p(x)
  \geq3r-2\sqrt{r^2/4+2nt}.
  \label{eq:wasserstein-lower}
\end{align}

Indeed, by \eqref{eq:squared-distance-heat}, $P_td_p(x),P_td_q(y)\leq\sqrt{r^2/4+2nt}$, whereas
$P_td_p(y)\geq3r-P_td_q(y)$.

Apply Proposition~\ref{prop:heat-transport} to 
$\gamma|_{[-r,r]}$. Since
$G_t(\dot\gamma,\dot\gamma)=1-D_t(\dot\gamma,\dot\gamma)$ and
$\sqrt{1-a}\leq1-a/2$ for $0\leq a\leq1$, \eqref{eq:wasserstein-upper} gives
\[
  W_2(\mu_{t,x},\mu_{t,y})
  \leq 2r-\frac12\int_{-r}^{r}
       D_t(\dot\gamma,\dot\gamma)\,ds.
\]
Comparison with \eqref{eq:wasserstein-lower} proves the first inequality in \eqref{eq:centered-estimate} and the second follows from elementary algebra. This completes the proof.
\end{proof}

 \section{Maximal trace and rigidity}

This section isolates the almost rigidity associated with maximal trace of $G_t$, which may be of independent interest. It will be used to show that a uniform deficit in the volume of unit balls forces a positive trace lower bound for $D_t=g-G_t$; see
Proposition~\ref{c1:prop:trace-gap}.  We first recall the scaling law of the
heat operator and the associated quadratic form.  If $g_r=r^{-2}g$, then
\begin{equation}\label{eq:heat-metric-scaling}
  P_s^{g_r}=P_{r^2s}^g,
  \qquad G_s^{g_r}=r^{-2}G_{r^2s}^g,
  \qquad
  \tr_{g_r}D_s^{g_r}=\tr_gD_{r^2s}^g.
\end{equation}

For $1\leq k\leq n$, define the rank-$k$ heat deficit by
\begin{equation}\label{eq:rank-k-heat-defect}
  \vartheta_k(t,x)
  =\min_{(e_1,\ldots,e_k)}\sum_{a=1}^kD_t(e_a,e_a),
\end{equation}
where the minimum is taken over orthonormal $k$-frames in $T_xM$.  Thus
$\vartheta_k(t,x)$ is the sum of the $k$ smallest eigenvalues of $D_t(x)$,
and $\vartheta_n(t,x)=\tr D_t(x)$.

The model for $\vartheta_k(t,x)=0$ is the rigidity of $2$-Bakry--\'{E}mery contraction inequality recorded by Ambrosio--Bru\`e--Semola
\cite{AmbrosioBrueSemola2019Rigidity} for weighted manifolds. 
The proposition below is the corresponding quantitative version encoded by $D_t$. Its rank-$n$ case gives Theorem~\ref{thm:maximal-trace-almost-rigidity},
while its rank-$(n-1)$ case is the ingredient for
Proposition~\ref{c2:prop:heat-gap}.

\begin{theorem}
\label{thm:maximal-trace-almost-rigidity}
For every $n\geq2$, $A>0$, $\eta_{GH}>0$, and
$\eta_{\mathrm{vol}}>0$, there is
$\delta=\delta(n,A,\eta_{GH},\eta_{\mathrm{vol}})>0$ with the following property.
If $(M^n,g)$ is complete with $\Ric_g\geq0$, $x\in M$, and $t>0$, then
\begin{equation}\label{eq:almost-maximal-trace}
  \tr_gG_t(x)\geq n-\delta
\end{equation}
implies
\begin{equation}\label{eq:almost-rigidity-gh}
  d_{GH}\!\left(
    B_g(x,A\sqrt t),
    B_{\mathbb R^n}(0,A\sqrt t)
  \right)<\eta_{GH}\sqrt t,
\end{equation}
and
\begin{equation}\label{eq:almost-rigidity-volume}
  \left|t^{-n/2}\vol_g B_g(x,A\sqrt t)-\omega_nA^n\right|
  <\eta_{\mathrm{vol}}.
\end{equation}
\end{theorem}
The rigidity theorem is a consequence of the following proposition, which is an almost version of the equality between the number of zero eigenvalues of $D_t$ at a certain point and the number of lines split off from $M$.
\begin{proposition}\label{prop:rank-k-heat-frame}
Let $1\leq k\leq n$, and let $(M_i^n,g_i,x_i)$ be pointed complete
Riemannian manifolds with $\Ric_{g_i}\geq0$.  If
\begin{equation}\label{eq:rank-k-defect-vanishes}
  \vartheta_k(1,x_i)\longrightarrow0,
\end{equation}
 then for every fixed $R\geq1$ there are harmonic functions
$h_i^1,\ldots,h_i^k$ on $B(x_i,4R)$ such that, for every
$1\leq a,b\leq k$,

\begin{align}\label{eq:harmonic-almost-splitting}
  \sup_{B(x_i,R)}&|\nabla h_i^a|\leq1+o_i(1)\\
  \dashint_{B(x_i,R)}
  &\left|\langle\nabla h_i^a,\nabla h_i^b\rangle-\delta^{ab}\right|^2
  \,d\vol_{g_i}\rightarrow0,\\
  \dashint_{B(x_i,R)}&|\nabla^2h_i^a|^2\,d\vol_{g_i}
  \rightarrow0.
\end{align}

Thus $h_i^1,\ldots,h_i^k$ are almost-splitting functions on $B(x_i,R)$.
\end{proposition}

\begin{proof}
Choose an orthonormal $k$-frame $e_i^1,\ldots,e_i^k$ realizing the
minimum in \eqref{eq:rank-k-heat-defect}.  From the supremum in
\eqref{eq:heat-metric}, choose $\psi_i^a\in C_c^\infty(M_i)$ and set
\[
  U_i^a=\nabla P_1\psi_i^a(x_i),
  \qquad b_i^a=P_1|\nabla\psi_i^a|^2(x_i).
\]
The functions may be chosen so that, with
$\rho_i=\vartheta_k(1,x_i)+i^{-1}$,
\begin{equation}\label{eq:almost-maximizers}
  \sum_{a=1}^k
  \bigl(|U_i^a-e_i^a|^2+b_i^a-|U_i^a|^2\bigr)
  \leq\rho_i.
\end{equation}
Indeed,
\[
  1-\bigl(2\langle U,e\rangle-b\bigr)
  =|U-e|^2+b-|U|^2
\]
for every unit vector $e$, and $b-|U|^2\geq0$ by the
Bakry--\'{E}mery estimate.  In particular,
\begin{equation}\label{eq:terminal-frame}
  \langle U_i^a,U_i^b\rangle\longrightarrow\delta^{ab}.
\end{equation}

For $0\leq s\leq1$, put $u_{i,s}^a=P_{1-s}\psi_i^a$ and define
\begin{equation}\label{eq:covariance}
  C_i(s)_{ab}
  =P_s\langle\nabla u_{i,s}^a,\nabla u_{i,s}^b\rangle(x_i)
   -\langle U_i^a,U_i^b\rangle.
\end{equation}
The Bakry--\'{E}mery estimate gives $C_i(s)\geq0$, while the polarized
Bochner formula gives
\begin{equation}\label{eq:covariance-interpolation}
  [C_i(\beta)-C_i(\alpha)]_{ab}
  =2\int_\alpha^\beta P_s\!\left(
    \langle\nabla^2u_{i,s}^a,\nabla^2u_{i,s}^b\rangle
    +\Ric(\nabla u_{i,s}^a,\nabla u_{i,s}^b)
  \right)(x_i)\,ds.
\end{equation}
Since $C_i(0)=0$, equations \eqref{eq:almost-maximizers} and
\eqref{eq:covariance} give
\begin{equation}\label{eq:terminal-covariance}
  0\leq\tr C_i(1)
  =\sum_{a=1}^k(b_i^a-|U_i^a|^2)
  \leq\rho_i.
\end{equation}
Taking the trace in \eqref{eq:covariance-interpolation} therefore yields
\[
  2\int_{1/2}^1\sum_{a=1}^kP_s\!\left(
    |\nabla^2u_{i,s}^a|^2
    +\Ric(\nabla u_{i,s}^a,\nabla u_{i,s}^a)
  \right)(x_i)\,ds
  \leq\tr C_i(1).
\]
Thus one can choose $s_i\in(1/2,1)$ such that, for
$u_i^a=u_{i,s_i}^a$,
\begin{equation}\label{eq:hessian-heat-small}
  \sum_{a=1}^kP_{s_i}\!\left(
    |\nabla^2u_i^a|^2
    +\Ric(\nabla u_i^a,\nabla u_i^a)
  \right)(x_i)
  \leq\tr C_i(1)\leq\rho_i.
\end{equation}
This explains the common choice of time in \eqref{eq:hessian-heat-small}.
Moreover, $0\leq C_i(s_i)\leq C_i(1)$, so
$\|C_i(s_i)\|\leq\tr C_i(1)$.  Combining this with
\eqref{eq:covariance} and \eqref{eq:terminal-frame} gives
\begin{equation}\label{eq:heat-frame-pointwise}
  P_{s_i}\langle\nabla u_i^a,\nabla u_i^b\rangle(x_i)
  \longrightarrow\delta^{ab}.
\end{equation}

We next pass from \eqref{eq:heat-frame-pointwise} to an averaged Gram
estimate.  For $c\in\mathbb R^k$, write
\[
  u_{i,c}=\sum_ac_au_i^a,
  \quad A_{i,c}=P_{s_i}|\nabla u_{i,c}|^2(x_i),
  \quad H_{i,c}=P_{s_i}|\nabla^2u_{i,c}|^2(x_i).
\]
Equations \eqref{eq:hessian-heat-small} and
\eqref{eq:heat-frame-pointwise} imply
$H_{i,c}\to0$ and $A_{i,c}\to|c|^2$.  Kato's inequality and the
semigroup Poincar\'{e} inequality give
\begin{equation}\label{eq:heat-poincare}
  \operatorname{Var}_{\mu_{s_i,x_i}}(|\nabla u_{i,c}|)
  \leq2s_iH_{i,c}.
\end{equation}
  One applies the inequality first
to $(|\nabla u_{i,c}|^2+\epsilon^2)^{1/2}$ and then lets
$\epsilon\downarrow0$.  If $X\geq0$, $A=E(X^2)$, and
$\kappa=\operatorname{Var}(X)$, then
\begin{align*}
  E|X^2-A|
  &\leq E|X^2-(EX)^2|+\bigl|A-(EX)^2\bigr|\\
  &\leq
  \bigl(E|X-EX|^2\bigr)^{1/2}
  \bigl(E|X+EX|^2\bigr)^{1/2}+\kappa\\
  &\leq2\sqrt{A\kappa}+\kappa.
\end{align*}
Indeed, $A-(EX)^2=\kappa$ and
$E|X+EX|^2=A+3(EX)^2\leq4A$.  Applying this with
$X=|\nabla u_{i,c}|$ and
$\kappa_{i,c}=\operatorname{Var}_{\mu_{s_i,x_i}}
(|\nabla u_{i,c}|)$ gives
\begin{align*}
  &\int_M\bigl||\nabla u_{i,c}|^2-|c|^2\bigr|\,
  d\mu_{s_i,x_i}\\
  &\quad\leq
  \int_M\bigl||\nabla u_{i,c}|^2-A_{i,c}\bigr|\,
  d\mu_{s_i,x_i}
  +|A_{i,c}-|c|^2|\\
  &\quad\leq
  2\sqrt{A_{i,c}\kappa_{i,c}}+\kappa_{i,c}
  +|A_{i,c}-|c|^2|
  \longrightarrow0.
\end{align*}
For $a=b$, take $c=e_a$.  For $a\ne b$, take $c=e_a+e_b$ and use
\[
  2\langle\nabla u_i^a,\nabla u_i^b\rangle
  =|\nabla(u_i^a+u_i^b)|^2-|\nabla u_i^a|^2-|\nabla u_i^b|^2,
\]
which yields
\begin{align*}
  2\int_M\left|
  \langle\nabla u_i^a,\nabla u_i^b\rangle
  \right|\,d\mu_{s_i,x_i} \leq\int_M\bigl||\nabla(u_i^a+u_i^b)|^2-2\bigr| +\bigl||\nabla u_i^a|^2-1\bigr|  +\bigl||\nabla u_i^b|^2-1\bigr|
     \,d\mu_{s_i,x_i}
  \rightarrow0.
\end{align*}
Thus, 
\begin{equation}\label{eq:gram-heat-l1}
  \int_M\left|
  \langle\nabla u_i^a,\nabla u_i^b\rangle-\delta^{ab}
  \right|\,d\mu_{s_i,x_i}\longrightarrow0.
\end{equation}

For fixed $R\geq1$, the Gaussian lower estimate for the heat kernel,
uniformly for $s_i\in[1/2,1]$, gives
\begin{equation}\label{eq:heat-lower}
  p_{s_i}(x_i,y)
  \geq\frac{c(n,R)}{\vol_{g_i} B(x_i,\sqrt{s_i})},
  \qquad y\in B(x_i,4R)
\end{equation}
\cite{LiYau1986ParabolicKernel}.  Since
$\vol_{g_i} B(x_i,\sqrt{s_i})\leq\vol_{g_i} B(x_i,4R)$,
\eqref{eq:hessian-heat-small} and \eqref{eq:gram-heat-l1} imply
\begin{align}
  \dashint_{B(x_i,4R)}
  \left|\langle\nabla u_i^a,\nabla u_i^b\rangle-\delta^{ab}\right|
  \,d\vol_{g_i}&\longrightarrow0,
  \label{eq:local-gram}\\
  \dashint_{B(x_i,4R)}|\nabla^2u_i^a|^2\,d\vol_{g_i}
  &\longrightarrow0.
  \label{eq:local-hessian}
\end{align}

Let $h_i^a$ be the harmonic replacement of $u_i^a$ in $B(x_i,4R)$.
Then $w_i^a=h_i^a-u_i^a\in W^{1,2}_0(B(x_i,4R))$.  Testing the
replacement equation with $w_i^a$, using the Dirichlet Poincar\'{e}
inequality and $|\Delta u_i^a|^2\leq n|\nabla^2u_i^a|^2$, gives
\begin{equation}\label{eq:harmonic-replacement}
  \dashint_{B(x_i,4R)}|\nabla w_i^a|^2\,d\vol_{g_i}
  \leq C(n)R^2
  \dashint_{B(x_i,4R)}|\nabla^2u_i^a|^2\,d\vol_{g_i}
  \longrightarrow0.
\end{equation}
Equations \eqref{eq:local-gram} and \eqref{eq:harmonic-replacement} give, for every $a,b$,
\begin{equation}\label{eq:harmonic-gram-l1}
  \dashint_{B(x_i,3R)}
  \left|\langle\nabla h_i^a,\nabla h_i^b\rangle-\delta^{ab}\right|
  \,d\vol_{g_i}\longrightarrow0.
\end{equation}
Indeed,
\[
  \left|\langle\nabla h_i^a,\nabla h_i^b\rangle
  -\langle\nabla u_i^a,\nabla u_i^b\rangle\right|
  \leq|\nabla w_i^a|\,|\nabla h_i^b|
      +|\nabla u_i^a|\,|\nabla w_i^b|,
\]
and the integral of the right side over $B(x_i,3R)$ tends to zero by
Cauchy--Schwarz, \eqref{eq:harmonic-replacement}, and the uniform Dirichlet energy
bounds on $B(x_i,4R)$.

For each $a$, the function $(|\nabla h_i^a|^2-1)_+$ is subharmonic.
The mean-value inequality and Bishop--Gromov comparison therefore give
\[
  \sup_{B(x_i,2R)}(|\nabla h_i^a|^2-1)_+
  \leq C(n,R)\dashint_{B(x_i,3R)}
  (|\nabla h_i^a|^2-1)_+\,d\vol_{g_i}
  \longrightarrow0.
\]
Together with \eqref{eq:harmonic-gram-l1}, this implies, componentwise,
\begin{equation}\label{eq:harmonic-gram-l2}
  \dashint_{B(x_i,2R)}
  \left|\langle\nabla h_i^a,\nabla h_i^b\rangle-\delta^{ab}\right|^2
  \,d\vol_{g_i}\longrightarrow0.
\end{equation}

Finally, choose a Laplacian cutoff $\chi_i$ supported in $B(x_i,2R)$,
equal to one on $B(x_i,R)$, and satisfying
$|\Delta\chi_i|\leq C(n)R^{-2}$.  The Bochner formula and
$\Ric_{g_i}\geq0$ give
\[
  2\int_{B(x_i,2R)}\chi_i|\nabla^2h_i^a|^2\,d\vol_{g_i}
  \leq\int_{B(x_i,2R)}
  (|\nabla h_i^a|^2-1)\Delta\chi_i\,d\vol_{g_i}
  =o_i(1)\vol_{g_i} B(x_i,R).
\]
This proves \eqref{eq:harmonic-almost-splitting}.  In particular, the harmonic
replacements are almost splitting maps.
\end{proof}

\begin{proof}[Proof of Theorem~\ref{thm:maximal-trace-almost-rigidity}]
After rescaling by $t^{-1}$, suppose that no such $\delta$ exists.  Then
there are pointed complete manifolds $(M_i^n,g_i,x_i)$ with
$\Ric_{g_i}\geq0$ and
\[
  \tr_{g_i}D_1(x_i)\longrightarrow0,
\]
while at least one of \eqref{eq:almost-rigidity-gh} and
\eqref{eq:almost-rigidity-volume}, with $t=1$, fails for every $i$.
Since $\vartheta_n(1,x_i)=\tr D_1(x_i)$,
Proposition~\ref{prop:rank-k-heat-frame} with $k=n$ produces harmonic splitting maps on every fixed ball. The almost-splitting to GH closeness theorem \cite{CheegerColding96} therefore gives
\begin{equation}\label{eq:local-almost-splitting}
  d_{GH}\!\left(B(x_i,A),B_{\mathbb R^n}(0,A)\right)
  \longrightarrow0.
\end{equation}
Colding's volume-convergence theorem \cite{Colding97} then gives,
\[
  \vol_{g_i} B(x_i,A)\longrightarrow\omega_nA^n.
\]
The two convergences contradict the choice of the
sequence. Scaling back by \eqref{eq:heat-metric-scaling} proves the theorem.
\end{proof}

\begin{remark}\label{rem:exact-maximal-trace}
If $\tr_gG_t(x)=n$ at one point and one positive time, then $(M,g)$ is
isometric to $\mathbb R^n$.  Indeed, let
$\eta_{GH},\eta_{\mathrm{vol}}\downarrow0$
and $A\uparrow\infty$ in
Theorem~\ref{thm:maximal-trace-almost-rigidity}.  More generally, the same argument shows that $\vartheta_k(t,x)=0$ produces splitting $\R^k$.
\end{remark}

\begin{remark}
    The time parameter $t$ plays no role under nonngenative Ricci curvature. However, if one allows negative Ricci curvature lower bound $K$, the same almost rigidity for positive eigenvalue holds at scale $\sqrt{t}\ll K^{-1}$. 
\end{remark}

We now recover the only consequence of maximal-trace almost rigidity needed
for the volume-growth theorem.

\begin{proposition}\label{c1:prop:trace-gap}
For every $n\geq2$ and $\epsilon>0$ for which the class below is nonempty,
there is $\delta=\delta(n,\epsilon)>0$ with the following property.  If
$(M^n,g)$ is complete with $\Ric_g\geq0$ and
\begin{equation}\label{eq:unit-volume-gap}
  \vol_gB_g(y,1)\leq\omega_n-\epsilon
  \qquad\text{for every }y\in M,
\end{equation}
then
\begin{equation}\label{eq:trace-gap}
  \tr D_{1/2}(x)\geq\delta
\end{equation}
for every $x\in M$.
\end{proposition}

\begin{proof}
Otherwise there are $(M_i,g_i,x_i)$ satisfying
\eqref{eq:unit-volume-gap} with $\tr D_{1/2}(x_i)\to0$.  Apply
Theorem~\ref{thm:maximal-trace-almost-rigidity} with $t=1/2$ and
$A=\sqrt2$.  Its volume conclusion gives
$\vol_{g_i} B(x_i,1)\to\omega_n$, contradicting
\eqref{eq:unit-volume-gap}.
\end{proof}

\begin{remark}\label{rem:delta-non-effective}
Equivalently, if $\mathcal C_{n,\epsilon}$ denotes the nonempty class of
pointed triples in Proposition~\ref{c1:prop:trace-gap}, then
\[
  \delta(n,\epsilon)
  =\inf_{(M,g,x)\in\mathcal C_{n,\epsilon}}\tr D_{1/2}(x)>0.
\]
\end{remark}
 \section{The codimension-two gap}
\label{c2:sec:heat-gap}

Order the eigenvalues of \(D_t(x)\), relative to \(g_x\), as
\begin{equation}
 d_1(t,x)\geq d_2(t,x)\geq\cdots\geq d_n(t,x)\geq0,
 \label{c2:eq:eigenvalues}
\end{equation}
and write
\begin{equation}
 q_t(x)=\tr D_t(x)-d_1(t,x)=\sum_{a=2}^nd_a(t,x).
 \label{c2:eq:heat-tail}
\end{equation}
Thus \(q_t\) is the deficit left after removing the largest eigenvalue.

\begin{proposition}
\label{c2:prop:heat-gap}
There is a constant \(\delta_n>0\), depending only on \(n\), such that
the following holds.
Let \((M^n,g)\), be a complete noncompact Riemannian $n$-manifold, $n\ge 3$. Suppose
\begin{equation}
 \Ric_g\geq0,
 \qquad \Scal_g\geq\sigma>0.
 \label{c2:eq:heat-gap-hypotheses}
\end{equation}
 Then, at the curvature time
\(t=\sigma^{-1}\),
\begin{equation}
 q_t(x)\geq\delta_n,
 \qquad
 d_2(t,x)\geq\rho_n:=\frac{\delta_n}{n-1}
 \quad\text{for every }x\in M.
 \label{c2:eq:heat-gap}
\end{equation}
\end{proposition}


The next Lemma is the new observation that allows one to show that uniformly positive scalar curvature prevents $(n-1)$-almost splitting. This was considered by Wang--Xie--B. Zhu--X. Zhu \cite{WangXieZhuZhu24}*{Remark 3.2}, and by Cucinotta--Mondino
\cite{CucinottaMondino26}*{Theorem~3.8(1)}  for uniformly positive integral scalar curvature lower bound. Both previous works require $M$ to be volume noncollapsed, to exclude the case when the positivity of scalar curvature is provided by only one direction in the Ricci tensor. The key observation is that the smallness of the Hessian can prevent the concentration of the positivity of Ricci tensor in one direction.
\begin{lemma}
\label{c2:lem:local-splitting}
There are \(A_n>1\) and \(\varepsilon_n>0\) such that the following
 holds. Let \((M^n,g,x)\) be a complete Riemannian manifold with
\(\Ric_g\geq0\) and \(\Scal_g\geq1\). Then there do not exist harmonic $\eps_n$-almost splitting functions \(h^1,\ldots,h^{n-1}\) in $(B(x,8A_n))$
satisfying, for every \(1\leq a,b\leq n-1\),
\begin{align} 
  \sup_{B(x,8A_n)}|\nabla h^a|
 &\leq1+\varepsilon_n.\label{c2:eq:harmonic-gradient}\\
 \frac1{\vol_g B(x,8A_n)}
 \int_{B(x,8A_n)}
 \left|\ip{\nabla h^a}{\nabla h^b}-\delta^{ab}\right|^2
 \dd\vol_g
&<\varepsilon_n, \label{c2:eq:local-Gram}\\
 \frac1{\vol_g B(x,8A_n)}
 \int_{B(x,8A_n)}
 |\nabla^2 h^a|^2\dd\vol_g
 &<\varepsilon_n, \label{c2:eq:local-Hess}
\end{align}
\end{lemma}

\begin{proof}
We first carry out the argument with a parameter \(A>1\), and choose
\(A=A_n\) at the end. A cutoff Bochner argument on
\(B(x,3A)\subset B(x,4A)\), using harmonicity and
\eqref{c2:eq:harmonic-gradient}, yields
\begin{equation}
 \int_{B(x,3A)}\sum_{a=1}^{n-1}
 \left(|\nabla^2 h^a|^2+\Ric(\nabla h^a,\nabla h^a)\right)\dd\vol_g
 \leq\frac{C_n}{A^2}\vol_g B(x,A).
 \label{c2:eq:local-Bochner}
\end{equation}

Put \(U_a=\nabla h^a\), let
\(H_{ab}=\ip{U_a}{U_b}\), and consider the ordinary
\((n-1)\)-form
\begin{equation}
 \beta=dh^1\wedge\cdots\wedge dh^{n-1}.
 \label{c2:eq:local-form}
\end{equation}
This definition is orientation-free and local.  Moreover, \(d\beta=0\),
\(|\beta|\leq C_n\), and
\begin{equation}
 |\nabla\beta|^2\leq C_n\sum_{a=1}^{n-1}|\nabla^2 h^a|^2
 \quad\text{on }B(x,4A)
 \label{c2:eq:form-energy}
\end{equation}
by \eqref{c2:eq:harmonic-gradient}.

Write \(\Delta_H=d\delta+\delta d\), let \(\nabla^*\nabla\) be the
nonnegative connection Laplacian, and let \(\mathcal R_{n-1}\) be the
curvature endomorphism in the Weitzenbock formula
\[
 \Delta_H=\nabla^*\nabla+\mathcal R_{n-1}
 \quad\text{on }(n-1)\text{-forms}.
\]
Choose \(\eta=1\) on \(B(x,2A)\), supported in \(B(x,3A)\), with
\(|\nabla\eta|\leq C/A\).  Apply the integrated Weitzenbock identity
to the compactly supported form \(\eta\beta\).  Since \(d\beta=0\),
\[
 d(\eta\beta)=d\eta\wedge\beta,
 \qquad
 \delta(\eta\beta)=\eta\,\delta\beta-\iota_{\nabla\eta}\beta,
 \qquad
 |\delta\beta|^2\leq n|\nabla\beta|^2.
\]
Hence the integrated Weitzenbock identity, \eqref{c2:eq:form-energy},
\eqref{c2:eq:local-Bochner}, and Bishop--Gromov comparison give
\begin{align}
 \int_M\eta^2
 \ip{\mathcal R_{n-1}\beta}{\beta}\dd\vol_g
 &=\int_M\left(
 |d(\eta\beta)|^2+|\delta(\eta\beta)|^2
 -|\nabla(\eta\beta)|^2\right)\dd\vol_g\notag\\
 &\leq C_n\int_{B(x,3A)}
 \left(|\nabla\eta|^2|\beta|^2
       +\eta^2|\nabla\beta|^2\right)\dd\vol_g\notag\\
 &\leq \frac{C_n}{A^2}\vol_g B(x,3A)
 +C_n\int_{B(x,3A)}\sum_{a=1}^{n-1}|\nabla^2 h^a|^2\dd\vol_g\notag\\
 &\leq\frac{C_n}{A^2}\vol_g B(x,A).
 \label{c2:eq:local-Hodge}
\end{align}
Choose a local orientation and set
\[
 N=(\star\beta)^\sharp.
\]
Although \(N\) changes sign if the orientation is reversed,
\(N^\flat\otimes N^\flat\) and \(\Ric(N,N)\) are independent of this
choice.  Since the Hodge star is parallel and commutes with the Hodge
Laplacian, the Weitzenbock formulas on one-forms and \((n-1)\)-forms give
\[
 \ip{\mathcal R_{n-1}\beta}{\beta}
 =\Ric(N,N).
\]
Moreover, \(N\) is orthogonal to each \(U_a\) and
\(|N|^2=|\beta|^2=\det(H)\).  The resulting tensor identity is
\[
 \det(H)g
 =N^\flat\otimes N^\flat
  +\sum_{a,b=1}^{n-1}\operatorname{cof}(H)_{ab}
    \,dh^a\otimes dh^b.
\]
Indeed, where \(H\) is invertible, the second term is \(\det(H)\) times
the metric on the span of the \(U_a\), while the first term is
\(\det(H)\) times the metric on its orthogonal complement.
The identity at singular \(H\) follows by continuity.  Contracting it
with the Ricci tensor gives
\begin{equation}
 \ip{\mathcal R_{n-1}\beta}{\beta}
 =\Ric(N,N)
 =\det(H)\Scal
  -\sum_{a,b=1}^{n-1}\operatorname{cof}(H)_{ab}\Ric(U_a,U_b).
 \label{c2:eq:cofactor-curvature}
\end{equation}
In particular, \(\Ric\geq0\) implies
\[
 \ip{\mathcal R_{n-1}\beta}{\beta}\geq0.
\]
Moreover, \(H\geq0\), hence
\(\operatorname{cof}(H)\geq0\), while
\(B=(\Ric(U_a,U_b))_{ab}\geq0\).  Therefore
\eqref{c2:eq:harmonic-gradient} implies
\[
 0\leq\operatorname{cof}(H)\leq C_nI,
 \qquad
 \sum_{a,b=1}^{n-1}\operatorname{cof}(H)_{ab}\Ric(U_a,U_b)
 \leq C_n\operatorname{tr}(B)
 =C_n\sum_{a=1}^{n-1}\Ric(U_a,U_a).
\]
Since \(\Scal\geq1\), it follows on \(B(x,4A)\) that
\begin{equation}
 \ip{\mathcal R_{n-1}\beta}{\beta}
 \geq\det(H)-C_n\sum_{a=1}^{n-1}\Ric(U_a,U_a).
 \label{c2:eq:scalar-cofactor}
\end{equation}

The matrices \(H\) are uniformly bounded by
\eqref{c2:eq:harmonic-gradient}, so the determinant is Lipschitz on their
range.  By Cauchy--Schwarz, \eqref{c2:eq:local-Gram}, and Bishop--Gromov
comparison,
\begin{equation}
 \frac1{\vol_g B(x,2A)}
 \int_{B(x,2A)}\det(H)\dd\vol_g
 =1+o_{\varepsilon}(1).
 \label{c2:eq:determinant-average}
\end{equation}
Here \(o_{\varepsilon}(1)\to0\) as \(\varepsilon\to0\), with \(A\) fixed.
Integrate \eqref{c2:eq:scalar-cofactor} on \(B(x,2A)\), use the
nonnegativity just noted to restrict \eqref{c2:eq:local-Hodge} to that ball,
and apply the Ricci part of \eqref{c2:eq:local-Bochner}.  Since
\(\vol_g B(x,2A)\geq\vol_g B(x,A)\), one obtains
\begin{equation}
 1-o_{\varepsilon}(1)\leq\frac{C_n}{A^2}.
 \label{c2:eq:local-contradiction}
\end{equation}
Choose \(A=A_n\) so large that the right side is less than \(1/4\),
and then choose \(\varepsilon_n\) so small that the left side is greater
than \(1/2\).  This contradiction proves the lemma.
\end{proof}

\begin{remark}
The positivity of Ricci tensor can concentrate in one direction. However, in this situation, the previous Lemma implies that the other directions are not close to splitting at large scales. Consider the Berger sphere $S^3$ with orthonormal frame $X_1,X_2,X_3$ so that $\sec(X_1,X_2)=-1$, $\sec(X_1,X_3)=1$, $\sec(X_2,X_3)=1$. This is possible by prescribing the commutator relations $[X_1,X_2]=2X_3$, $[X_2,X_3]=X_1$, and $[X_3,X_1]=X_2$. Then $\Ric(X_3,X_3)=1$ and $\Ric(X_1,X_1)=\Ric(X_2.X_2)=0$. However, if $\theta^i\defeq X_i^*$, then $\beta\defeq \theta^1\wedge\theta^2=\star\theta^3$ satisfies $|\nabla \beta|^2=2$. This $\beta$ is not almost parallel so there are no splitting functions whose gradient are almost parallel to $X_1$, $X_2$. 
\end{remark}

\begin{proof}[Proof of Proposition~\ref{c2:prop:heat-gap}]
By constant rescaling it suffices to take \(\sigma=1\) and \(t=1\).
Indeed, the relative eigenvalues satisfy
\[
 d_j^{\,\sigma g}(1,x)=d_j^{\,g}(\sigma^{-1},x),
 \qquad 1\leq j\leq n.
\]
Set \(A=A_n\) as in Lemma~\ref{c2:lem:local-splitting}.

Suppose the conclusion fails. There are pointed complete noncompact manifolds
\((M_i,g_i,x_i)\) with
\[
 \Ric_{g_i}\geq0,\qquad
 \Scal_{g_i}\geq1,\qquad
 q_1(x_i)\longrightarrow0.
\]
Moreover,
\[
  q_1(x_i)=\vartheta_{n-1}(1,x_i)
\]
by \eqref{eq:rank-k-heat-defect} and
\eqref{c2:eq:eigenvalues}--\eqref{c2:eq:heat-tail}.  Apply
Proposition~\ref{prop:rank-k-heat-frame} with $k=n-1$ and $R=8A$.
For all large $i$, the resulting harmonic functions satisfy
\eqref{c2:eq:local-Gram}, \eqref{c2:eq:local-Hess}, and
\eqref{c2:eq:harmonic-gradient}, contradicting
Lemma~\ref{c2:lem:local-splitting}.  Hence \(q_1\) has a positive
dimensional infimum \(\delta_n\) on the normalized class.
Finally, \(d_2\) is the largest of \(d_2,\ldots,d_n\), so
\(d_2\geq q_1/(n-1)\).  Scaling back proves
\eqref{c2:eq:heat-gap}.
\end{proof}
 \section{Liouville averaging and Proof of main theorems}
\label{sec:angular-pairs}
In this section we prove Theorem \ref{thm:codim-one} and Theorem \ref{thm:codim-two} together based on common average estimates through Liouville measure. Similar arguments can be found in \cite{CaiVolume}*{Lemma 3.1, 3.2}. Throughout this section, $(M^n,g)$ is a smooth complete Riemannian manifold with
$\Ric_g\geq0$, and
\begin{equation}\label{eq:k-range}
  k\in\{1,2\}.
\end{equation}
In fact the argument in this section is valid for a full range of $k$ from $1$ through $n$, since positive curvature conditions in our setting will only force $1$ or $2$ positive eigenvalues $D_t$  we restrict our discussion to $k=1,2$. With stronger positive intermediate curvature, one can improve the estimates of number of positive eigenvalues of $D_t$ through splitting argument and the argument below will provide corresponding volume growth order bounds and effective constants, see \cites{antonelli2026universalvolumegrowthbounds, koirala2026volumegrowthpositiveintermediate}. 

Let $\omega_x$ be the surface measure on $S_xM$.  The following elementary estimate is the only place where the number of deficit directions
enters.

\begin{lemma}
\label{lem:angular-k}
There is a constant $c(n,k)>0$ with the following property.  If $A$ is a
nonnegative quadratic form on $\mathbb R^n$ with at least $k$ eigenvalues
not smaller than $\rho>0$, then every measurable
$F\subset S^{n-1}$ satisfies
\begin{equation}\label{eq:angular-k}
  \int_F A(v,v)\dd\omega(v)
  \geq c(n,k)\rho\,\omega(F)^{1+\frac2k}.
\end{equation}
\end{lemma}

\begin{proof}
Let $\Pi$ be the span of $k$ orthonormal eigenvectors corresponding to
eigenvalues at least $\rho$.  Then
\[
  A(v,v)\geq\rho|\operatorname{proj}_{\Pi}v|^2.
\]
After a rotation, write
\[
  v=(\sin\theta\,\xi,\cos\theta\,\eta),
  \qquad
  \xi\in S^{k-1},\quad
  \eta\in S^{n-k-1},\quad
  0\leq\theta\leq\frac{\pi}{2}.
\]
Then $|\operatorname{proj}_{\Pi}v|=\sin\theta$, and spherical measure has
density
\[
  (\sin\theta)^{k-1}(\cos\theta)^{n-k-1}
  \dd\theta\dd\omega(\xi)\dd\omega(\eta).
\]
Since $\arcsin a\leq(\pi/2)a$ for $0<a\leq1$, integration over
$0\leq\theta\leq\arcsin a$ gives a finite constant $b(n,k)$, which we
enlarge so that $b(n,k)\geq\omega(S^{n-1})$, with
\begin{equation}\label{eq:tubular-band}
  \omega\{v\in S^{n-1}:|\operatorname{proj}_{\Pi}v|\leq a\}
  \leq b(n,k)a^k
  \qquad(0<a\leq1).
\end{equation}
If $\omega(F)=0$, there is nothing to prove.  Otherwise set
\[
  a=\left(\frac{\omega(F)}{2b(n,k)}\right)^{1/k}\leq1.
\]
By \eqref{eq:tubular-band} and the choice of $a$, the part of $F$ outside the band $\{|\operatorname{proj}_{\Pi}v|\leq a\}$ has measure at least $\omega(F)/2$. Outside the band,
$A(v,v)\geq\rho a^2$, and therefore
\[
  \int_FA(v,v)\dd\omega(v)
  \geq\frac12\rho a^2\omega(F)
  =\frac{\rho}{2(2b(n,k))^{2/k}}\omega(F)^{1+\frac2k}.
\]
This proves \eqref{eq:angular-k} with
$c(n,k)=[2(2b(n,k))^{2/k}]^{-1}$.
\end{proof}

Let $\pi:SM\to M$ be the footpoint map, let $\varphi_s$ be geodesic
flow, and give $SM$ Liouville measure
\[
  d\lambda=d\vol_g\dd\omega.
\]
For a Borel set $\Omega\subset M$, set
\[
  S\Omega=\pi^{-1}(\Omega),
  \qquad
  \Omega_a=\{z\in M:d(z,\Omega)<a\}.
\]
If $v\in SM$, write $\gamma_v(s)=\pi(\varphi_sv)$ and define
\begin{equation}\label{eq:minimizing-directions-k}
  \mathcal G_L
  =\{v\in SM:\gamma_v|_{[-L/2,L/2]}\text{ is minimizing}\},
  \qquad
  Q_\Omega(L)=\lambda(S\Omega\cap\mathcal G_L).
\end{equation}
Note that $\mathcal G_L$ is Borel.

\begin{proposition}\label{prop:liouville-tail-k}
Fix $t,\rho>0$.  Suppose that $D_t(x)$ has at least $k$ eigenvalues not
smaller than $\rho$ at every $x\in M$.  If $\Omega\subset M$ is Borel,
$L>0$, and
$\vol_g(\Omega_{L/3})<\infty$, then
\begin{equation}\label{eq:liouville-tail-k}
  Q_\Omega(L)
  \leq K_{n,k}\left(\frac{t}{\rho}\right)^{k/2}
       \frac{\vol_g(\Omega_{L/3})}{L^k},
  \qquad
  K_{n,k}=3^k\left(\frac{4n}{c(n,k)}\right)^{k/2}.
\end{equation}
\end{proposition}

\begin{proof}
Write $L=3r$, set $E=S\Omega\cap\mathcal G_{3r}$. Integrating Proposition~\ref{prop:centered-segment} over $E$
gives
\begin{equation}\label{eq:integrated-centered-budget-k}
  \int_{-r}^{r}\int_E
  D_t(\varphi_sv,\varphi_sv)\dd\lambda(v)\dd s
  \leq\frac{8nt}{r}\lambda(E).
\end{equation}
For fixed $s\in(-r,r)$, Liouville invariance gives
$\lambda(\varphi_sE)=\lambda(E)$, and all footpoints of $\varphi_sE$ lie in
$\Omega_r$.  Let
\[
  q_s(z)=\int_{S_zM}\mathbf1_{\varphi_sE}(v)\dd\omega_z(v).
\]
Then $q_s$ vanishes almost everywhere off $\Omega_r$ and
\begin{equation}\label{eq:fiber-mass-k}
  \int_{\Omega_r}q_s\dd\vol_g=\lambda(E).
\end{equation}
Lemma~\ref{lem:angular-k} applied fiberwise, followed by H\"older's inequality, yields
\begin{align}
\int_{\varphi_sE}D_t(v,v)\dd\lambda(v)&=\int_{\Omega_r}\int_{\varphi_sE\cap S_zM}D_t(v,v)\dd\omega_z(v)\dd\vol_g(z)\notag\\
  &\geq c(n,k)\rho
    \int_{\Omega_r}q_s(z)^{1+2/k}\dd\vol_g(z)\notag\\
  &\geq c(n,k)\rho\,
    \frac{\lambda(E)^{1+2/k}}{\vol_g(\Omega_r)^{2/k}}.
  \label{eq:fiber-lower-k}
\end{align}
If $\lambda(E)=0$, the conclusion is immediate.  Otherwise integrate \eqref{eq:fiber-lower-k} over the interval
$(-r,r)$, compare with \eqref{eq:integrated-centered-budget-k}, cancel
$\lambda(E)$, and raise the resulting inequality to the power $k/2$.  Since
$r=L/3$, this gives exactly \eqref{eq:liouville-tail-k}.
\end{proof}

For almost every ordered pair $(p,q)$ there is a unique minimizing geodesic. Let $m(p,q)$ be its midpoint. This is to align with the symmetric parametrization of geodesic that has been used in Proposition \ref{prop:centered-segment} and \eqref{eq:minimizing-directions-k}. The cut relation and the diagonal
are null in the product volume measure.  
For Borel $\Omega\subset
M$, define
\begin{equation}\label{eq:pair-set-k}
  \mathcal E_\Omega(R)
  =\{(p,q):d(p,q)<R,\ m(p,q)\in\Omega\}.
\end{equation}

\begin{proposition}
\label{prop:ordered-pairs-k}
Under the hypotheses of Proposition~\ref{prop:liouville-tail-k}, if
$\Omega\subset M$ is Borel, $R>0$, and
$\vol_g(\Omega_{R/3})<\infty$, then
\begin{equation}\label{eq:ordered-pair-k}
  \vol_{M\times M}(\mathcal E_\Omega(R))
  \leq K_{n,k}\left(\frac{t}{\rho}\right)^{k/2}
       \int_0^R\ell^{n-1-k}\vol_g(\Omega_{\ell/3})\dd\ell.
\end{equation}
The integral is finite at the origin because $n>k$.
\end{proposition}

\begin{proof}
For $v\in SM$, let
\[
  c(v)=\sup\{a>0:\gamma_v|_{[0,a]}\text{ is minimizing}\}
\]
be the forward cut time.  For $0<\ell<c(v)$, let
$A_v(\ell):v^\perp\to\dot\gamma_v(\ell)^\perp$ be the normal Jacobi
tensor with $A_v(0)=0$ and $A_v'(0)=I$.  Define the endpoint polar
Jacobian, extended through the cut locus, by
\[
  \overline J(v,\ell)=
  \begin{cases}
    |\det A_v(\ell)|,&0<\ell<c(v),\\
    0,&\ell\geq c(v),
  \end{cases}
\]
and set $\overline J(v,0)=0$.  Endpoint polar coordinates give, for every
nonnegative measurable $f$,
\begin{equation}\label{eq:endpoint-polar-k}
  \int_{M\times M}f(p,q)\dd\vol_g(p)\dd\vol_g(q)
  =\int_{SM}\int_0^\infty
  f(\pi v,\pi\varphi_\ell v)\overline J(v,\ell)
  \dd\ell\dd\lambda(v).
\end{equation}
At fixed $\ell$, recenter by $w=\varphi_{\ell/2}v$.  Liouville
invariance preserves $d\lambda$, so, with
\[
  \overline J_{\mathrm{sym}}(w,\ell)
  =\overline J(\varphi_{-\ell/2}w,\ell),
\]
equation \eqref{eq:endpoint-polar-k} becomes
\begin{equation}\label{eq:midpoint-polar-k}
  \vol_{M\times M}(\mathcal E_\Omega(R))
  =\int_{S\Omega}\int_0^R
    \overline J_{\mathrm{sym}}(w,\ell)\dd\ell\dd\lambda(w).
\end{equation}
Radial Bishop comparison gives
\begin{equation}\label{eq:jacobian-bound-k}
  0\leq\overline J_{\mathrm{sym}}(w,\ell)
  \leq\ell^{n-1}\mathbf1_{\mathcal G_\ell}(w).
\end{equation}
Combining
\eqref{eq:midpoint-polar-k}, \eqref{eq:jacobian-bound-k}, and Proposition~\ref{prop:liouville-tail-k} proves
\eqref{eq:ordered-pair-k}.
\end{proof}

\begin{proof}[Proof of Theorem~\ref{thm:codim-one}]
Let $\delta=\delta(n,\epsilon)>0$ be supplied by
Proposition~\ref{c1:prop:trace-gap}.  Since $D_{1/2}\geq0$ and
$\tr D_{1/2}\geq\delta$, its largest eigenvalue is at least
\begin{equation}\label{eq:codim-one-rho}
  \rho=\frac{\delta}{n}.
\end{equation}
Thus Propositions~\ref{prop:liouville-tail-k} and
\ref{prop:ordered-pairs-k} apply with $k=1$ and $t=1/2$.

Fix $o\in M$, write $V(R)=\vol_g B(o,R)$, take
$\Omega=B(o,2R)$, and use pair-distance bound $2R$.  Almost every pair in
$B(o,R)^2$ lies in $\mathcal E_\Omega(2R)$.  For
$0<\ell<2R$, Bishop--Gromov comparison gives
\begin{equation}\label{eq:codim-one-ball-neighborhood}
  \Omega_{\ell/3}\subset B(o,8R/3),
  \qquad
  V(8R/3)\leq(8/3)^nV(R).
\end{equation}
Consequently
\begin{align}
  V(R)^2
  &\leq K_{n,1}\left(\frac{n}{2\delta}\right)^{1/2}
     \left(\frac83\right)^nV(R)
     \int_0^{2R}\ell^{n-2}\dd\ell\notag\\
  &=C(n,\epsilon)R^{n-1}V(R).
  \label{eq:codim-one-cancellation}
\end{align}
Every ball of positive radius has finite positive volume by Hopf--Rinow.
Cancelling $V(R)$ proves Theorem~\ref{thm:codim-one}. 
\end{proof}

\begin{proof}[Proof of Theorem~\ref{thm:codim-two}]
Fix $o\in M$ and write $V(R)=\vol_g B(o,R)$.  Let $\delta_n$ be
the dimensional constant from Proposition~\ref{c2:prop:heat-gap}.  Proposition~\ref{c2:prop:heat-gap} says that,
at time $t=\sigma^{-1}$, the second eigenvalue of $D_t$ is at least
\begin{equation}\label{eq:codim-two-rho}
  \rho_n=\frac{\delta_n}{n-1}.
\end{equation}
Apply Proposition~\ref{prop:ordered-pairs-k} with $k=2$,
$\Omega=B(o,2R)$, and pair-distance bound $2R$.  As above, almost every
pair in $B(o,R)^2$ is counted, and
\begin{equation}\label{eq:codim-two-ball-neighborhood}
  \Omega_{\ell/3}\subset B(o,8R/3),
  \qquad
  V(8R/3)\leq(8/3)^nV(R)
  \quad(0<\ell<2R).
\end{equation}
Since $n>2$,
\begin{align}
  V(R)^2
  &\leq K_{n,2}\frac{\sigma^{-1}}{\rho_n}
    \left(\frac83\right)^nV(R)
    \int_0^{2R}\ell^{n-3}\dd\ell\notag\\
  &\leq C_n\sigma^{-1}R^{n-2}V(R).
  \label{eq:codim-two-cancellation}
\end{align}
Cancelling $V(R)$ proves Theorem~\ref{thm:codim-two}.
\end{proof}
 
\appendix
\section{Codimension-one volume estimates for sectional curvature}
Gromov recorded the following codimension-one estimate under nonnegative
sectional curvature with only a sketch of the strategy \cite{Gro85}. We give a direct proof of it with different techniques. 

\begin{theorem}[Gromov]\label{thm:Gromov}
Let $(X,\dist)$ be a complete $n$-dimensional Alexandrov space with
nonnegative curvature.  Assume that there is a real number
$\delta\in(0,1)$ such that
\[
  \haus^n(B_1^X(x))\leq(1-\delta)\omega_n
  \qquad\text{for every }x\in X,
\]
where $\omega_n$ is the volume of the Euclidean unit ball in $\mathbb R^n$.
Then there exists a constant $C=C(n,\delta)$ such that, for every $p\in X$
and every $R\geq1$,
\[
  \haus^n(B_R^X(p))\leq C R^{n-1}.
\]
\end{theorem}

We use two ingredients from Alexandrov geometry.  The first is the following
quantitative form of volume rigidity, a standard consequence of compactness
and volume convergence for Alexandrov spaces \cite{BGP92}.

\begin{proposition}\label{prop:VolRigidity}
For every $n$ and $\delta>0$, there exists
$\epsilon=\epsilon(n,\delta)>0$ with the following property.  If
$X\in\mathrm{Alex}^n(0)$ and, for some pointed complete length space
$(Z,z)$,
\[
  \dist_{GH}\bigl(B_\rho^X(x),
    B_\rho^{\mathbb R^n\times Z}((0,z))\bigr)
  <\epsilon\rho,
\]
then
\[
  \haus^n(B_\rho^X(x))>(1-\delta)\omega_n\rho^n.
\]
\end{proposition}

The second is the Li--Naber packing theorem
\cite{LiNaber2020SingularSets}.

\begin{theorem}[Li--Naber]\label{thm:packing}
Let $X\in\mathrm{Alex}^n(-1)$.  For every $n$ and $\epsilon>0$, there are
constants
\[
  C_p=C_p(n,\epsilon),
  \qquad
  \beta=\beta(n,\epsilon)\in(0,1]
\]
such that the following holds.  If
\[
  x_i\in\mathcal S_{\epsilon,\beta r_i}^k(X)\cap B_1(q),
  \qquad
  B_{r_i}(x_i)\text{ are pairwise disjoint},
  \qquad
  r_i\leq1,
\]
then
\[
  \sum_i r_i^k\leq C_p.
\]
Here $\mathcal S_{\epsilon,r}^k(X)$ is the $(\epsilon,k)$-effective stratum
at scale $r$: a point $x$ belongs to this stratum when $B_r(x)$ is not
$\epsilon r$-Gromov--Hausdorff close to a ball in
$\mathbb R^{k+1}\times Z$ for any metric space $Z$.
\end{theorem}

We apply this with $k=n-1$.  Although the singular set of an Alexandrov
space without boundary has codimension two, its
$(\epsilon,n-1)$-effective stratum may be nonempty at a fixed scale.
For $A\subset Y$, let $P_s^Y(A)$ denote the supremum of the cardinalities of
pairwise disjoint families of balls of radius $s$ in $Y$ with centers in $A$.

\begin{proof}[Proof of Theorem~\ref{thm:Gromov}]
Fix $p\in X$ and $R\geq1$.  Let
\[
  (X_R,\dist_R,p):=(X,R^{-1}\dist,p).
\]
Thus $B_1^{X_R}(p)=B_R^X(p)$ and the Hausdorff measures scale as
\begin{equation}\label{eq:appendix-measure-scaling}
  \haus_{X_R}^n=R^{-n}\haus_X^n.
\end{equation}
For every $x\in X$, the unit-ball deficit and
\eqref{eq:appendix-measure-scaling} give
\[
  \haus_{X_R}^n(B_{R^{-1}}^{X_R}(x))
  \leq(1-\delta)\omega_nR^{-n}.
\]
The contrapositive of Proposition~\ref{prop:VolRigidity} therefore yields
\begin{equation}\label{eq:appendix-effective-stratum}
  X_R\subset\mathcal S_{\epsilon,R^{-1}}^{n-1}(X_R)
\end{equation}
for some $\epsilon=\epsilon(n,\delta)>0$.

Let $\beta=\beta(n,\epsilon)$ be supplied by
Theorem~\ref{thm:packing}.  We first assume $R\geq\beta^{-1}$ and set
\[
  r=\frac1{\beta R}\leq1.
\]
If $\{B_r^{X_R}(x_i)\}_{i\in I}$ is any pairwise disjoint family with
centers in $B_1^{X_R}(p)$, then
\eqref{eq:appendix-effective-stratum} says that
\[
  x_i\in\mathcal S_{\epsilon,R^{-1}}^{n-1}(X_R)
  =\mathcal S_{\epsilon,\beta r}^{n-1}(X_R).
\]
The Li--Naber theorem gives
\[
  |I|r^{n-1}\leq C_p,
\]
and hence
\begin{equation}\label{eq:appendix-packing-rescaled}
  P_r^{X_R}(B_1^{X_R}(p))
  \leq C_p\beta^{n-1}R^{n-1}.
\end{equation}
Scaling back, this is precisely
\begin{equation}\label{eq:appendix-packing-original}
  P_{1/\beta}^X(B_R^X(p))
  \leq C_p\beta^{n-1}R^{n-1}.
\end{equation}

Choose a maximal $2/\beta$-separated set $\{x_i\}_{i\in I}$ in
$B_R^X(p)$.  The balls $B_{1/\beta}^X(x_i)$ are pairwise disjoint, so
\eqref{eq:appendix-packing-original} bounds $|I|$.  Maximality also implies
that the balls $B_{2/\beta}^X(x_i)$ cover $B_R^X(p)$.  Bishop--Gromov
comparison now gives
\begin{align*}
  \haus^n(B_R^X(p))
  &\leq\sum_{i\in I}\haus^n(B_{2/\beta}^X(x_i))\\
  &\leq\omega_n(2/\beta)^n|I|
  \leq C(n,\delta)R^{n-1}.
\end{align*}

It remains only to treat $1\leq R<\beta^{-1}$.  In this range,
Bishop--Gromov comparison directly gives
\[
  \haus^n(B_R^X(p))\leq\omega_nR^n
  \leq\omega_n\beta^{-1}R^{n-1}.
\]
After enlarging the constant, the two ranges prove the theorem.
\end{proof}
 
\bibliographystyle{alpha}
\bibliography{ref}

\end{document}